\documentclass[12pt]{amsart}
\usepackage{amssymb,amscd,amsthm,amsmath,color}
\usepackage{fullpage}

\newcommand{\PP}{\mathbb{P}}
\newcommand{\OO}{\mathcal{O}}

\newcommand{\uU}{\mathcal{U}}

\newcommand{\Sym}{\operatorname{Sym}}

\theoremstyle{plain}
\newtheorem{lemma}{Lemma}[section]
\newtheorem*{theorem*}{Theorem}
\newtheorem*{lemma*}{Lemma}
\newtheorem*{proposition*}{Proposition}
\newtheorem*{conjecture*}{Conjecture}
\newtheorem*{corollary*}{Corollary}
\newtheorem*{problem*}{Problem}
\newtheorem{theorem}[lemma]{Theorem}

\newtheorem{corollary}[lemma]{Corollary}
\newtheorem{proposition}[lemma]{Proposition}

\theoremstyle{definition}
\newtheorem{definition}[lemma]{Definition}

\newtheorem{remark}[lemma]{Remark}

\begin{document}

\title{Effective bounds on ampleness of cotangent bundles}
\author[I. Coskun]{Izzet Coskun}
\address{Department of Mathematics, Statistics and CS \\UIC, Chicago, IL 60607}
\email{coskun@math.uic.edu}

\author[E. Riedl]{Eric Riedl}
\address{Department of Mathematics, 255 Hurley, Notre Dame, IN 46556 }
\email{eriedl@nd.edu}

\subjclass[2010]{Primary: 14M10. Secondary: 14F10, 32C38}
\keywords{Complete intersections, ample cotangent bundle}
\thanks{During the preparation of this article the first author was partially supported by the  NSF grant DMS-1500031 and NSF FRG grant  DMS 1664296 and  the second author was partially supported  by the NSF RTG grant DMS-1246844.}

\begin{abstract}
We prove that a general complete intersection of dimension $n$, codimension $c$ and type $d_1, \dots, d_c$ in $\PP^N$ has ample cotangent bundle if $c \geq 2n-2$ and the $d_i$'s are all greater than a bound that is $O(1)$ in $N$ and quadratic in $n$. This degree bound substantially improves the currently best-known super-exponential bound in $N$ by Deng, although our result does not address the case $n \leq c < 2n-2$.
\end{abstract}

\maketitle

\section{Introduction}
Let $X$ be a general complete intersection in $\PP^N$ of dimension $n>1$ and type $d_1, \dots, d_c$. In this note, we prove that if $c \geq 2n-2$ and 
 $$d_i \geq   \frac{(2n-2)(24n-28)}{N-3n+3}+2,$$ then the cotangent bundle $\Omega_X$ is ample.

Debarre  conjectured that a general complete intersection $X \subset \PP^N$ with $c \geq n$ has ample cotangent bundle provided that the degrees $d_i$ defining $X$ are sufficiently large \cite{Debarre}. Debarre's Conjecture has been proven by both Brotbek and Darondeau \cite{BrotbekDarondeau} and Xie \cite{Xie}. Brotbek and Darondeau do not provide effective bounds, while Xie showed that one can take $d_i \geq N^{N^2}$ to guarantee that $\Omega_X$ is ample \cite{Xie}. Deng in \cite{Deng1, Deng2} improved the bounds to $d_i \geq 16c^2 (2N)^{2N+2c}$. 

When $c \geq 2n-2$, our bounds are vast improvements on these exponential bounds. In fact, our bound is 3 as soon as $N \geq 48n^2 -101n +53$. In earlier work, Brotbek \cite{BrotbekExp} proved that if $c \geq 3n -2$ and all the degrees are equal $d_i =d$, then $\Omega_X$ is ample provided that $d \geq 2N+3$. While our bound is less restrictive on $c$ and is better for $N$ large with respect to $n$, in the case $c \geq 3n-2$, $d_i = d$ for all $i$, and $N$ small relative to $n$, Brotbek's bound of $2N+3$ is better. Finally, Brotbek in \cite{BrotbekCot} showed that a general complete intersection surface has ample $\Omega_X$ if $d_i \geq \frac{8N+2}{N-3}$. 

The first step is to clarify and improve Brotbek's \cite{BrotbekCot} estimates that guarantee that $\Omega_X$ is ample outside a codimension 2 subvariety. We use a more careful combinatorial analysis and a theorem of Darondeau. This sets up a new application of a technique of Riedl and Yang \cite{RiedlYang, RiedlYang2}, which allows us to remove the non-ample locus. This process loses a little on the codimension bound relative to \cite{BrotbekDarondeau}, but gives much better bounds on the degrees.

\subsection*{Organization of the paper} In \S \ref{sec-prelim}, following Brotbek \cite{BrotbekCot}, we obtain degree bounds that guarantee that $\Omega_X$ is ample outside a variety of codimension 2. In \S \ref{sec-main}, using the technique of Riedl and Yang \cite{RiedlYang, RiedlYang2}, we show how to remove the non-ample locus.

\subsection*{Acknowledgments} We thank Damian Brotbek, Lionel Darondeau, Lawrence Ein, Mihai P\u{a}un and David Yang for enlightening discussions.

\section{Ampleness outside a codimension $2$ set}\label{sec-prelim}
Let $X \subset \PP^N$ be a general complete intersection of dimension $n$ and type $d_1, \dots, d_c$. We always assume that the codimension $c = N-n \geq n$. Let $\Omega_X$ denote the cotangent bundle of $X$ and let $$\pi: \PP \Omega_X \rightarrow X$$ be the natural projection. In this section, we give bounds on the degrees $d_i$ that guarantee that $\Omega_X$ is ample outside of a codimension $2$ set. We follow the basic strategy from Brotbek \cite{BrotbekCot} closely.  However, using a more careful analysis of the combinatorics and a new theorem of Darondeau, we improve his  bounds, which are exponential in $n$, significantly.  

\begin{definition}
Let $E$ be a vector bundle on a projective variety $Y$ and let $H$ be an ample line bundle on $Y$. Let $\pi: \PP(E) \rightarrow Y$ denote the projection. If for some $\epsilon$ with $0< \epsilon \ll 1$, any irreducible curve $C \subset \PP(E)$ with $C \cdot \OO_{\PP(E)}(1)  < \epsilon \  C \cdot \pi^* H$ satisfies $\pi(C) \subset T$,  then $E$ is said to be {\em ample outside $T \subset Y$}.
\end{definition}

It follows from the definition that if $\Sym^k E$ is globally generated outside of a subvariety $T$ of $Y$ for some $k > 0$, then $E \otimes H$ is ample outside of $T$. In \cite{BrotbekCot}  Brotbek proves the following theorem.

\begin{theorem}\cite[Theorem 4.5, Corollary 4.7]{BrotbekCot}\label{thm-BrotbekCot}
Let $X \subset \PP^N$ be a general complete intersection of dimension $n$ and type $d_1, \dots, d_c$. If $\OO_{\PP \Omega_X}(1) \otimes \pi^* \OO_{X}(-a-N)$ on $\PP \Omega_X$ is big, then the projection of the stable  base locus of $\OO_{\PP \Omega_X}(1) \otimes \OO(-a)$ under $\pi$ has codimension at least 2 in $X$. Thus, if $\OO_{\PP \Omega_X}(1) \otimes \pi^* \OO_{X}(-N-1)$ is big, then $\Omega_X$ is ample outside an algebraic set $Y$ of codimension at least 2 in $X$, where $Y$ is the image under $\pi$ of the stable base locus of $\OO_{\PP \Omega_X}(1) \otimes \pi^* \OO_{X}(-1)$.
\end{theorem}

We now explain how a theorem of Darondeau allows us to remove the dependence on $N$ in Theorem \ref{thm-BrotbekCot}.  Let $\PP^{N_1} \times \dots \times \PP^{N_c}$ be the moduli space of all tuples of homogeneous polynomials $(f_1, \dots, f_c)$ of degrees $d_1, \dots, d_c$, respectively. Let $B \subset  \PP^{N_1} \times \dots \times \PP^{N_c}$ be the  Zariski open subset parameterizing tuples that intersect transversely and thus define smooth complete intersections of type $d_1, \dots, d_c$. Let  $\uU$ be the universal family over $B$, whose points parametrize tuples $(p,f_1,\dots,f_c)$ where $p \in V(f_1, \dots, f_c)$.

\begin{theorem}[Main Theorem, compact case from \cite{Darondeau}]
The vector bundle $T_{\PP(\Omega_{\uU/B})} \otimes \OO_{\PP^N}(3) \otimes \OO_{B} (1,\dots, 1)$ is globally generated.
\end{theorem}

By replacing Merker's bound (Theorem 4.9 in \cite{BrotbekCot}) with Darondeau's improved bound from \cite{Darondeau} in the proof of Theorem 4.5 in \cite{BrotbekCot}, one obtains the following.

\begin{theorem}\label{thm-OurCot}
Let $X \subset \PP^N$ be a general complete intersection of dimension $n$ and type $d_1, \dots, d_c$. If $\OO_{\PP \Omega_X}(1) \otimes \pi^* \OO_{X}(-a-3)$ on $\PP \Omega_X$ is big, then the projection of the stable  base locus of $\OO_{\PP \Omega_X}(1) \otimes \OO_X(-a)$ under $\pi$ has codimension at least 2 in $X$. Thus,  if $\OO_{\PP \Omega_X}(1) \otimes \pi^* \OO_{X}(-4)$ is big, then $\Omega_X$ is ample outside an algebraic set $Y$ of codimension at least 2 in $X$, where $Y$ is the image under $\pi$ of the stable base locus of $\OO_{\PP \Omega_X}(1) \otimes \pi^* \OO_{X}(-1)$.
\end{theorem}

In view of Theorem \ref{thm-OurCot}, we desire effective bounds on the degrees $d_i$ that guarantee that the line bundles  $\OO_{\PP \Omega_X}(1) \otimes \pi^* \OO_{X}(-3)$ and $\OO_{\PP \Omega_X}(1) \otimes \pi^* \OO_{X}(-4)$ are big. Recall the following  criterion for bigness of a line bundle.

\begin{theorem}\cite[Theorem 2.2.15]{LazarsfeldBook}\label{thm-bigness}
If $F$ and $G$ are nef line bundles on an $r$-dimensional variety and $F^r > r F^{r-1} \cdot G$, then $F-G$ is big.
\end{theorem}

We will use the following proposition from Brotbek.

\begin{proposition}\cite[Proposition 4.2]{BrotbekCot}\label{prop-omega2}
Let $Y \subset \PP^N$ be a smooth projective variety. The bundle $\Omega_Y(2)$ is ample if and only if  $Y$ does not contain lines.
\end{proposition}

\begin{theorem}\label{thm-big}
Let $X \subset \PP^N$ be a smooth complete intersection of dimension $n$ and type $d_1,\dots, d_c$ with $c \geq n$. Let $a \geq -1$ be an integer.
If  $$ d_i \geq  \frac{n((2n-1)(a+2) +2 )}{N-2n+1}  + 2$$
 for all $i$, then $\OO_{\PP\Omega_X}(1) \otimes \pi^* \OO_X(-a)$ is big.
\end{theorem}
\begin{proof}
Under our assumptions on $d_i$, the general complete intersection $X$ does not contain any lines. Consequently, by Brotbek's Proposition \ref{prop-omega2}, $\Omega_X(2)$ is ample. Equivalently, the line bundle $F= \OO_{\PP \Omega_X}(1) \otimes \pi^* \OO_X(2)$ is ample on $\PP \Omega_X$.  The line bundle $G = \pi^* \OO_{X}(a+2)$ is nef on $\PP \Omega_X$ being the pullback of a nef line bundle on $X$. By Theorem \ref{thm-bigness},  $\OO_{\PP \Omega_X}(1) \otimes \pi^* \OO_{X}(-a)$ is big if $$F^{2n-1} > (2n-1) F^{2n-2} \cdot G.$$

Recall that the Segre classes of a rank $r$ vector bundle $E$ are defined by $$s_i(E) = \pi_* ((c_1(\OO_{\PP E}(1))^{r-1+i}).$$ Thus, $F^{2n-1} = s_{n}(\Omega_X(2))$ and by push-pull, $F^{2n-2} \cdot G = s_{n-1}(\Omega_X(2)) \cdot (a+2)H$. Let $s(E)$ denote the total Segre class of $E$.

The Euler sequence on $\PP^N$ twisted by $\OO_{\PP^N}(2)$ 
\[ 0 \to \Omega_{\PP^N}(2) \to \OO(1)^{N+1} \to \OO(2) \to 0 \]
implies that
\[ s(\Omega_{\PP^N}(2)) = \frac{1-2H}{(1-H)^{N+1}}. \]  
The conormal sequence for $X$
\[ 0 \to \bigoplus_{i=1}^{c} \OO(-d_i+2) \to \Omega_{\PP^N}(2)|_X \to \Omega_X(2) \to 0  \]
yields
\[ s(\Omega_X(2)) = \frac{(1-2H)\prod_{i=1}^c(1+(d_i-2)H)}{(1-H)^{N+1}} .\]

Let $\epsilon_k(x_1, \dots, x_{c}) = \sum_{i_1< \cdots < i_k} x_{i_1} \cdots x_{i_k}$ denote the $k$th elementary symmetric function in $x_1, \dots, x_{c}$. For an $r$-tuple $d=(d_1, \dots, d_c)$, let $$\phi_{k,d} = \epsilon_k(d_1-2, \dots, d_c -2).$$
 Since $$\frac{1}{(1-x)^{N+1}} = \frac{d^N}{dx^N} \left( \frac{1}{N!} \frac{1}{1-x} \right) =  \frac{d^N}{dx^N} \left(\frac{1}{N!} \sum_{i \geq 0} x^i\right) = \sum_{i\geq 0} \binom{i+N}{N} x^i,$$ we obtain the relation
\[  s(\Omega_X(2)) = (1-2H) \left(\sum_{i \geq 0} \phi_{i,d} H^i \right) \left( \sum_{i \geq 0} \binom{i+N}{N} H^i \right) .\]
For our purposes, we only need $s_{n}(\Omega_X(2))$ and $s_{n-1}(\Omega_X(2))$. Then 
\[ s(\Omega_X(2))= (1-2H)(\dots + b_{n-2} H^{n-2} + b_{n-1}H^{n-1} + b_{n}H^{n}), \]
where
\[ b_{n} = \sum_{k=0}^{n} \phi_{k,d} \binom{N+n-k}{N} \]
\[ b_{n-1} = \sum_{k=0}^{n-1} \phi_{k,d} \binom{N+n-k-1}{N} \]
\[ b_{n-2} = \sum_{k=0}^{n-2} \phi_{k,d} \binom{N+n-k-2}{N} .\]
Then we have $$s_{n} = (b_{n} - 2b_{n-1})H^{n} \quad \mbox{and} \quad s_{n-1} = (b_{n-1} - 2b_{n-2})H^{n-1}.$$

We would like to determine when $s_{n} - (2n-1)(a+2)s_{n-1}$ is positive. This quantity equals
\[b_{n} - ((2n-1)(a+2) +2 )b_{n-1} + 2(2n-1)(a+2)b_{n-2} . \]
Expanding out this expression using the convention that $\phi_{k,d}=0$ for $k < 0$, we obtain 
\begin{equation}\label{eq1}
\sum_{k=0}^n \binom{N+n-k}{N} \left( \phi_{k,d} - ((2n-1)(a+2)+2) \phi_{k-1,d} + 2(2n-1)(a+2) \phi_{k-2, d} \right)
\end{equation}
This quantity is positive if 
$$\frac{\phi_{k,d}}{\phi_{k-1,d}} \geq  ((2n-1)(a+2)+2)$$ for all $1 \leq k \leq n$. 
Lemma \ref{lem-symmetricPolys} shows that  $$\frac{\phi_{k,d}}{\phi_{k-1,d}} \geq \frac{c-k+1}{k} \min_i\{d_i -2\}. $$
Hence, the quantity (\ref{eq1}) is positive if
$$\frac{c-k+1}{k} \min_i\{d_i -2\} \geq ((2n-1)(a+2)+2)$$ for all $1 \leq k \leq n$. 
Recalling that $c= N-n$, this inequality is satisfied for $1 \leq k \leq n$ when 
$$d_i \geq \frac{n}{N-2n+1} ((2n-1)(a+2) +2 ) + 2. $$ This concludes the proof of the theorem modulo the proof of Lemma \ref{lem-symmetricPolys}.
\end{proof}

\begin{lemma}
\label{lem-symmetricPolys}
Let $k < r$ and let $x_i$ be positive real numbers. Then the following inequality holds
\[ \frac{\epsilon_k(x_1,\dots,x_r)}{\epsilon_{k-1}(x_1, \dots, x_r)} \geq  \frac{r-k+1}{k} \min_i\{ x_i \} . \]
\end{lemma}
\begin{proof}
First, we show that the quotient $\epsilon_k / \epsilon_{k-1}$ is an increasing function in $x_i$.  This allows us to replace all of the $x_i$ with $\min \{x_i\}$.  Recall that $$\frac{\partial}{\partial x_i} \epsilon_k(x_1, \dots, x_r) = \epsilon_{k-1}(x_1, \dots, \hat{x}_i, \dots, x_r).$$ For simplicity, denote $\epsilon_k(x_1, \dots, x_r)$ by $\epsilon_k$ and $\epsilon_k(x_1, \dots, \hat{x}_i, \dots, x_r)$ by $\hat{\epsilon}_{k, i}$. 
Hence, 
\[ \frac{\partial}{\partial x_i} \frac{\epsilon_k(x_1,\dots,x_r)}{\epsilon_{k-1}(x_1, \dots, x_r)} = \frac{\epsilon_{k-1}\hat{\epsilon}_{k-1,i} - \epsilon_k \hat{\epsilon}_{k-2, i}}{\epsilon_{k-1}^2} .\]
We would like to show this quantity is positive. It suffices to show the numerator is positive. We compute the coefficient of $\prod_{j=1}^r x_j^{a_j}$ in $\epsilon_{k-1}\hat{\epsilon}_{k-1,i}$ and $\epsilon_k \hat{\epsilon}_{k-2, i}$. 

First, both coefficients are zero  unless $$0 \leq a_j \leq 2 \ \mbox{ for all} \ j \not= i,\  0 \leq a_i \leq 1, \ \mbox{ and} \  \sum_{j=1}^{r} a_j = 2k-2.$$ Let $S$ be the set of $j$ such that $a_j =2$ and let $|S|=m$. Let $I \subset \{1, \dots, r\}$ be the set of $j$ such that $a_j = 1$. 

If $i \in I$, then the coefficient of $\prod_{j=1}^r x_j^{a_j}$ in $\epsilon_{k-1}\hat{\epsilon}_{k-1,i}$ is given by $\binom{2k-3 - 2m}{k-2-m}$. This is the number of ways of writing $\prod_{j=1}^r x_j^{a_j}$ as a product of two monomials $m_1 m_2$ of length $k_1$ such that the terms in $m_1$ and $m_2$ are all distinct and $x_i \mid m_1$. Since the terms in $m_1$ and $m_2$ are distinct, $x_j | m_1$ and $x_j | m_2$ for $j \in S$. Hence, the coefficient is given by the number of ways of choosing $k-2-m$ elements in $I \backslash \{i\}$. 

Similarly, if $i \in I$, the coefficient of $\prod_{j=1}^r x_j^{a_j}$ in $\epsilon_k \hat{\epsilon}_{k-2, i}$ is given by $\binom{2k-3 - 2m}{k-1-m}$. This corresponds to choosing $k-1-m$ elements out of $I \backslash \{i\}$. Hence, when $i \in I$, the coefficients $\prod_{j=1}^r x_j^{a_j}$ in $\epsilon_{k-1}\hat{\epsilon}_{k-1,i}$ and  $\epsilon_k \hat{\epsilon}_{k-2, i}$ are equal.

By similar reasoning, if $i \not\in I$, then the coefficient of $\prod_{j=1}^r x_j^{a_j}$ in $\epsilon_{k-1}\hat{\epsilon}_{k-1,i}$ is given by $\binom{2k-2 - 2m}{k-1-m}$, with the convention that $\binom{0}{0}=1$. The coefficient of $\prod_{j=1}^r x_j^{a_j}$ in $\epsilon_k \hat{\epsilon}_{k-2, i}$ is given by $\binom{2k-2 - 2m}{k-m}$. Since $\binom{2k-2 - 2m}{k-1-m}> \binom{2k-2 - 2m}{k-m},$ we conclude that the numerator is positive. 

Hence, the quotient  $\frac{\epsilon_k(x_1,\dots,x_r)}{\epsilon_{k-1}(x_1, \dots, x_r)}$ increases as $x_i$ increases. Let $x = \min\{x_i\}$. Hence, we get a lower bound for the quotient by setting each of the $x_i=x$. We obtain $\epsilon_k(x,\dots, x) = \binom{r}{k} x^k$. This gives
\[ \frac{\epsilon_k(x_1,\dots,x_r)}{\epsilon_{k-1}(x_1, \dots, x_r)} \geq \frac{\binom{r}{k}x^k}{\binom{r}{k-1}x^{k-1}} = \frac{r-k+1}{k} x  \] This concludes the proof of the lemma.
\end{proof}

Combining Theorem \ref{thm-OurCot} and Theorem \ref{thm-big}, we obtain the following corollary.

\begin{corollary}\label{cor-mainbig}
Let $X \subset \PP^N$ be a general complete intersection of dimension $n$ and type $d_1,\dots, d_c$ with $c \geq n$.  If 
$$d_i \geq \frac{(2n^2-n)(a+5)+2n}{N-2n+1}+2$$ for all $i$, then the projection of the stable base locus of $\OO_{\PP \Omega_X}(1) \otimes \OO(-a)$ has codimension at least 2 in $X$. In particular, if 
  $$d_i \geq \frac{12n^2-4n}{N-2n+1}+2$$ for all $i$, then the projection of the stable base locus of $\OO_{\PP \Omega_X}(1) \otimes \pi^* \OO_X(-1)$ has codimension at least 2 in $X$, which implies $\Omega_X$ is ample outside a variety of codimension at least 2 in $X$.
\end{corollary}

\section{Ampleness everywhere}\label{sec-main}
In this section, using a technique of Riedl and Yang introduced in \cite{RiedlYang} and further developed in \cite{RiedlYang2}, we remove the base locus at the expense of slightly worse bounds. 

For simplicity, let $d= (d_1, \dots, d_c)$. Let $\uU_{N,d}$ denote an open subvariety of the universal complete intersection parameterizing pairs $(p,X)$, where $X$ is a complete intersection in $\PP^N$ of dimension $n$ and type  $d_1, \dots, d_c$ and $p$ is a point of $X$. The main tool is the following theorem of Riedl and Yang.

\begin{theorem}\cite[Theorem 2.3]{RiedlYang2}
\label{thm-Grassmann}
Let $M$ and $t$ be  positive integers. Suppose that for every $N$, we have a countable union of locally closed subvarieties $Z_{N,d} \subset \uU_{N,d}$ satisfying the following two conditions:
\begin{enumerate}
\item The codimension of $Z_{M,d}$ in $\uU_{M,d}$ is at least $t$.
\item If $(p, X_0) \in Z_{N-1,d}$ is a linear section of $(p,X) \in \uU_{N,d}$, then $(p,X) \in Z_{N,d}$.
\end{enumerate}
Then for any $u \geq 0$, $Z_{M-u,d} \subset \uU_{M-u,d}$ has codimension at least $u+t$.
\end{theorem}

Applying Theorem \ref{thm-Grassmann} to Corollary \ref{cor-mainbig}, we can obtain the main result of this note.

\begin{theorem}
Let $X \subset \PP^N$ be a general complete intersection of dimension $n$ and type $d_1, \dots, d_c$, and suppose $n > 1$.
\begin{enumerate}
\item If  $c \geq 2n-1$, $a \geq -1$ and 
\[ d_i \geq \frac{(8n^2-10n+3)a+ 40n^2 -46n+13}{N-3n+2} + 2 \]
for all $i$, then the stable base locus of $\OO_{\PP\Omega_X}(1) \otimes  \pi^*\OO_X(-a)$ is empty and some multiple is globally generated.
\item  If $c \geq 2n-2$ and
$$d_i \geq  \frac{(2n-2)(24n-28)}{N-3n+3}+2$$
for all $i$, then $\Omega_X$ is ample.
\end{enumerate}
\end{theorem}

If $X$ is a curve,  $\Omega_X$ is a line bundle of degree $-N-1+\sum_i d_i$, which is globally generated if $\sum_i d_i \geq N+1$ and ample if $\sum_i d_i > N+1$.

\begin{proof}
By Corollary \ref{cor-mainbig}, if $M \geq 2m$ and 
$$d_i \geq \frac{(2m^2 -m)(a+5) + 2m}{M-2m+1}+ 2$$ then for some $k \gg 0$ we have $\Sym^k(\Omega_X)(-ka)$ is globally generated outside a subvariety of codimension at least $2$ for a general complete intersection $X \subset \PP^M$ of dimensinon $m$. Let $\uU_{N,d}$ be the subvariety of the universal complete intersection of type $d_1, \dots, d_c$ in $\PP^N$ consisting of pairs $(p,X)$ such that all sections of $\Sym^k(\Omega_X)(-ka)$ extend to the general complete intersection. Let $Z_{N,d}$ be the locus of points $(p,X)$ where $\Sym^k(\Omega_X)(-ka)$  is not globally generated. When $N=M$, $Z_{M,d}$ has codimension at least 2 in $\uU_{M,d}$, so satisfies (1) in Theorem \ref{thm-Grassmann} with $t=2$. 
Combining the restriction sequence $$0 \rightarrow \Omega_X(-1) \rightarrow \Omega_X \rightarrow \Omega_X|_{X \cap H} \rightarrow 0$$ and the conormal sequence $$0 \rightarrow \OO_{X\cap H} (-1) \rightarrow \Omega_X|_{X\cap H} \rightarrow \Omega_{X \cap H} \rightarrow 0,$$ we see that there is a surjective map $\Omega_X \rightarrow \Omega_{X \cap H} \rightarrow 0$. Consequently, we obtain a surjective map  $\Sym^k(\Omega_X)(-ka) \rightarrow \Sym^k(\Omega_{X \cap H})(-ka)$. Hence, if the latter is not globally generated at $p$, the former is certainly not globally generated at $p$ either. Hence, $Z_{N,d}$ satisfies  (2) in Theorem \ref{thm-Grassmann}. We conclude that $Z_{M - u, d}$ has codimension at least $u+2$ in $\uU_{M-u, d}$. If $u+2 > m-u$, then the projection of $Z_{M - u, d}$ to the space of complete intersections cannot be dominant. Letting $N= M-u$, $n= m-u$, $u=n-1$ and substituting into the degree bounds for $d_i$, we obtain the first statement.

Similarly, by Corollary \ref{cor-mainbig}, if
$$d_i \geq \frac{12m^2 -4m}{M-2m+1} + 2,$$ then $\Omega_X$ is ample outside a subvariety of codimension at least $2$. Let $Z_{N,d}$ be the locus of points $(p,X)$ where $\Omega_X$ fails to be ample. Then for $N=M$ this locus has codimension at least 2, so satisfies (1) in Theorem \ref{thm-Grassmann} with $t=2$. The surjection  $\Omega_X \rightarrow \Omega_{X \cap H} \rightarrow 0$ induces a map $\PP \Omega_{X \cap H} \rightarrow \PP \Omega_X$ such that the restriction of $\OO_{\PP \Omega_X}(1)$ to the image coincides with $\OO_{\PP \Omega_{X\cap H}}(1)$. Consequently, given a curve $C \in X\cap H$ passing through $p$ satisfying $\OO_{\PP \Omega_{X\cap H}}(1) \cdot C < \epsilon \ \pi^* H \cdot C$, the same curve satisfies  $\OO_{\PP \Omega_{X}}(1) \cdot C < \epsilon \ \pi^* H \cdot C$. Hence, $Z_{N,d}$ satisfies  (2) in Theorem \ref{thm-Grassmann}. We conclude that $Z_{M - u, d}$ has codimension at least $u+2$ in $\uU_{M-u, d}$. If $u+2 \geq m-u$, then the projection of $Z_{M - u, d}$ to the space of complete intersections cannot be dominant. If it were dominant, then the fibers would be finite. However, if the fibers are nonempty, then they have to be at least 1 dimensional since they contain curves. Letting $N= M-u$, $n=m-u$ and $u= n-2$ and substituting into the degree bounds for $d_i$, we obtain the second statement. Note that taking $u=n-2$ in this last step requires $n > 1$.
\end{proof}

\begin{corollary}\label{cor-fantastic}
Assume that $N \geq 48n^2 -101n +53$. Then the general complete intersection of dimension $n$ in $\PP^N$ of type $d_1, \dots, d_{N-n}$ has ample cotangent bundle if $d_i \geq 3$.
\end{corollary}

\begin{remark}
Inspired by the case of curves, one could speculate that a complete intersection of dimension $n$ and type $d_1, \dots, d_c$ in $\PP^N$ will have ample cotangent bundle if $d_i \geq 2$ provided that $c \gg n$. 
\end{remark}

\bibliographystyle{plain}

\end{document}